\documentclass[a4paper]{article}

\usepackage{hyperref} 

\usepackage{amssymb,amsmath,amsthm}
\usepackage[english]{babel}

\usepackage{todonotes}
\usepackage{verbatim} 
\usepackage{enumerate}
\usepackage{mathtools}

\usepackage{bbm}

\usepackage{cite}

\newcommand{\cM}{\mathcal{M}}
\newcommand{\cN}{\mathcal{N}}

\newcommand{\cT}{\mathcal{T}}

\newcommand{\bR}{\mathbb{R}}

\newcommand{\dd}{ \mathrm{d}}

\renewcommand{\epsilon}{\varepsilon}

\newcommand{\ip}[2]{\langle #1,#2\rangle}

\numberwithin{equation}{section}

\newtheorem{theorem}{Theorem}[section]
\newtheorem{lemma}[theorem]{Lemma}
\newtheorem{proposition}[theorem]{Proposition}
\newtheorem{corollary}[theorem]{Corollary}

\theoremstyle{definition}
\newtheorem{definition}[theorem]{Definition}

\newtheorem{remark}[theorem]{Remark}

\usepackage{tikz}
\usetikzlibrary{matrix}

\begin{document}

\title{A Banach-Dieudonn\'{e} theorem for the space of bounded continuous functions on a separable metric space with the strict topology}

\author{
\renewcommand{\thefootnote}{\arabic{footnote}}
Richard Kraaij
\footnotemark[1]
}

\footnotetext[1]{
Delft Institute of Applied Mathematics, Delft University of Technology, Mekelweg 4, 
2628 CD Delft, The Netherlands, E-mail: \texttt{r.c.kraaij@tudelft.nl}.
}

\maketitle

\begin{abstract}
Let $X$ be a separable metric space and let $\beta$ be the strict topology on the space of bounded continuous functions on $X$, which has the space of $\tau$-additive Borel measures as a continuous dual space. We prove a Banach-Dieudonne\'{e} type result for the space of bounded continuous functions equipped with $\beta$. As a consequence, this space is hypercomplete and a Pt\'{a}k space. Additionally, the closed graph, inverse mapping and open mapping theorems holds for linear maps between space of this type.
\end{abstract}

{\bf Mathematics Subject Classifications (2010).} 46E10 (primary);
46E27 (secondary)

{\bf{Key words.} Banach-Dieudonne theorem; space of bounded continuous functions; strict topology; closed graph theorem;}



\section{Introduction and main result}

Let $(E,t)$ be a locally convex space. Denote by $E'$ the continuous dual space of $(E,t)$ and denote by $\sigma = \sigma(E',E)$ the weak topology on $E'$. We consider the following additional topologies on $E'$:
\begin{itemize}
\item $\sigma^f$, the finest topology coinciding with $\sigma$ on all $t$-equi-continuous sets in $E'$. 
\item $\sigma^{lf}$, the finest locally convex topology coinciding with $\sigma$ on all $t$-equi-continuous sets in $E'$.
\item $t^\circ$ the polar topology of $t$ defined on $E'$. $t^\circ$ is defined in the following way. Let $\cN$ be the collection of all $t$ pre-compact sets in $E$. A pre-compact set, is a set that is compact in the completion of $(E,t)$. Then the topology $t^\circ$ on $E'$ is generated by all seminorms of the type
\begin{equation*}
p_N(\mu) := \sup_{f \in N} |\ip{f}{\mu}| \qquad N \in \cN.
\end{equation*}
\end{itemize}

The Banach-Dieudonn\'{e} theorem for locally convex spaces is the following, see Theorems 21.10.1 and 21.9.8 in \cite{Ko69}.

\begin{theorem}[Banach-Dieudonn\'{e}] \label{theorem:BD}
Let $(E,t)$ be a metrizable locally convex space, then the topologies $\sigma^f$ and $t^\circ$ on $(E,t)'$ coincide. If $(E,t)$ is complete, these topologies also coincide with $\sigma^{lf}$.
\end{theorem}

The Banach-Dieudonn\'{e} theorem is of interest in combination with the closed graph theorem. For the discussion of closed graph theorems, we need some additional definitions. Considering a locally convex space $(E,t)$, we say that  
\begin{enumerate}[(a)]
\item $E'$ satisfies the \textit{Krein-Smulian} property if every $\sigma^f$ closed absolutely convex subset of $E'$ is $\sigma$ closed;
\item $(E,t)$ is a \textit{Pt\'{a}k space} if every $\sigma^f$ closed linear subspace of $E'$ is $\sigma$ closed;
\item $(E,t)$ is a \textit{infra Pt\'{a}k space} if every $\sigma$ dense $\sigma^f$ closed linear subspace of $E'$ equals $E'$.
\end{enumerate}

Infra-Pt\'{a}k spaces are sometimes also called \textit{$B_r$ complete} and Pt\'{a}k spaces are also known as \textit{$B$ complete} or \textit{fully complete}. Finally, a result of \cite{Ke58} shows that the Krein-Smulian property for $E'$ is equivalent to \textit{hypercompleteness} of $E$: the completeness of the space of absolutely convex closed neighbourhoods of $0$ in $(E,t)$ equipped with the Hausdorff uniformity. 

Clearly, we have that $E$ hypercomplete implies $E$ Pt\'{a}k implies $E$ infra Pt\'{a}k. Additionally, if $E$ is a infra Pt\'{a}k space, then it is complete by 34.2.1 in \cite{Ko79}. See also Chapter 7 in \cite{PCB87} for more properties of Pt\'{a}k spaces.

We have the following straightforward result, using that the absolutely convex closed sets agree for all locally convex topologies that give the same dual.

\begin{proposition} \label{proposition:hypercomplete_implied_by_sigma_is_sigmaf}
If $\sigma^{lf}$ and $\sigma^f$ coincide on $E'$, then $E$ is hypercomplete and infra Pt\'{a}k.
\end{proposition}

This last property connects the Banach-Dieudonn\'{e} theorem to the closed graph theorem.

\begin{theorem}[Closed graph theorem, cf. 34.6.9 in \cite{Ko79}] \label{theorem:closed_graph}
Every closed linear map of a barrelled space $E$ to an infra-Pt\'{a}k space $F$ is continuous.
\end{theorem}


As Banach and Fr\'{e}chet spaces are metrizable, they are infra-Pt\'{a}k space by Theorem \ref{theorem:BD} and Proposition \ref{proposition:hypercomplete_implied_by_sigma_is_sigmaf}. Additionally, they are both barrelled spaces, which implies that the closed graph theorem holds for closed linear maps from a Fr\'{e}chet space to a Fr\'{e}chet space. As a consequence, also the inverse and open mapping theorems hold. 

\smallskip

In this paper, we study the space of bounded and continuous functions on a separable metric space $X$ equipped with the \textit{strict} topology $\beta$. For the definition and a study of the properties of $\beta$, see \cite{Se72}. The study of the strict topology is motivated by the fact that it is a `correct' generalization of the supremum norm topology on $C_b(X)$ from the setting where $X$ is compact to the setting that $X$ is non-compact. Most importantly, for the strict topology, the dual space equals the space $\cM_\tau(X)$ of $\tau$-additive Borel measures on $X$. A Borel measure $\mu$ is called $\tau$-additive if for any increasing net $\{U_\alpha\}_{\alpha}$ of open sets, we have
\begin{equation*}
\lim_{\alpha} |\mu|(U_\alpha) = |\mu|\left(\cup_\alpha U_\alpha\right).
\end{equation*}

In the case that $X$ is metrizable by a complete separable metric, the space of $\tau$ additive Borel measures equals the space of Radon measures. Additionally, in this setting $(C_b(X),\beta)$ satisfies the Stone-Weierstrass theorem, cf. \cite{Ha76}, and the Arzela-Ascoli theorem.

\smallskip

The space $(C_b(X),\beta)$ is not barrelled unless $X$ is compact, Theorem 4.8 of \cite{Se72} so Theorem \ref{theorem:closed_graph} does not apply for this class of spaces. Thus, the following closed graph theorem by Kalton is of interest, as it puts more restrictions on the spaces serving as a range, relaxing the conditions on the spaces allowed as a domain.

\begin{theorem}[Kalton's closed graph theorem, Theorem 2.4 in \cite{Ka71}, Theorem 34.11.6 in \cite{Ko79}] \label{theorem:Kalton}
Every closed linear map from a Mackey space $E$ with weakly sequentially complete dual $E'$ into a transseparable infra-Pt\'{a}k space $F$ is continuous.
\end{theorem}

\begin{remark}
Note that this result is normally stated for separable infra-Pt\'{a}k space $F$. In the proof of Kalton's closed graph theorem 34.11.6 in \cite{Ko79}, separability is only used to obtain that weakly compact sets of the dual $E'$ are metrizable. For this transseparability suffices by Lemma 1 in \cite{Pf76}.
\end{remark}

A class of spaces, more general than the class of Fr\'{e}chet spaces, satisfying the conditions for both the range and the domain space in Kalton's closed graph theorem, would be an interesting class of spaces to study. In this paper, we show that $(C_b(X),\beta)$, for a separable metric space $X$ belongs to this class. In particular, the main result in this paper is that $(C_b(X),\beta)$ satisfies the conclusions of the Banach-Dieudonn\'{e} theorem.

First, we introduce an auxiliary result and the definition of a $k$-space. 

\begin{proposition} \label{proposition:Cb_strong_Mackey}
$(C_b(X),\beta)$ is a strong Mackey space. In other words, $\beta$ is a Mackey topology and the weakly compact sets in $\cM_\tau(X)$ and the weakly closed $\beta$ equi-continuous sets coincide.
\end{proposition}

\begin{proof}
This follows by Theorem 5.6 in \cite{Se72}, Corollary 6.3.5 and Proposition 7.2.2(iv) in \cite{Bo07}.
\end{proof}

This result is relevant in view of the defining properties of $\sigma^f$. We say that a topological space $(Y,t)$ is a k-space if a set $A \subseteq Y$ is $t$-closed if and only if $A \cap K$ is $t$-closed for all $t$-compact sets $K \subseteq Y$. The strongest topology on $Y$ coinciding on $t$-compact sets with the original topology $t$ is denoted by $kt$ and is called the $k$-ification of $t$. The closed sets of $kt$ are the sets $A$ in $Y$ such that $A \cap K$ is $t$-closed in $Y$ for all $t$-compact sets $K \subseteq Y$.

We see that for a strong Mackey space $E$, $\sigma^f = k\sigma$ on $E'$.

\smallskip

The main result of this paper is that $(C_b(X),\beta)$ also satisfies the conclusion of the Banach-Dieudonn\'{e} theorem. 

\begin{theorem} \label{theorem:main_theorem}
Let $X$ be a separably metrizable space. Consider the space $(C_b(X),\beta)$, where $\beta$ is the strict topology. Then $\sigma^{lf}$, $\sigma^f$, $k\sigma$ and $\beta^\circ$ coincide on $\cM_\tau(X)$.
\end{theorem}

In view of Kalton's closed graph theorem, we mention two additional relevant results, that will be proven below.

\begin{lemma} \label{lemma:transseparable}
Let $X$ be a separably metrizable space. Then $(C_b(X),\beta)$ is transseparable.
\end{lemma}

\begin{lemma} \label{lemma:weak_sequential_completeness}
Let $X$ be separably metrizable, then the dual $\cM_\tau(X)$ of $(C_b(X),\beta)$ is weakly sequentially complete.
\end{lemma}

As a consequence of Theorem \ref {theorem:main_theorem} and Lemma's \ref{lemma:transseparable} and \ref{lemma:weak_sequential_completeness}, $(C_b(X),\beta)$ satisfies both the conditions to serve as a range, and as a domain in Kalton's closed graph theorem. We have the following important corollaries.

\begin{corollary}[Closed graph theorem] \label{corollary:closed_graph}
Let $X, Y$ be separably metrizable spaces, then a closed linear map from $(C_b(X),\beta)$ to $(C_b(Y),\beta)$ is continuous.
\end{corollary}

\begin{corollary}[Inverse mapping theorem] \label{corollary:inverse_mapping}
Let $X, Y$ be separably metrizable spaces. Let $T : (C_b(X),\beta) \rightarrow (C_b(Y),\beta)$ be a bijective continuous linear map. Then $T^{-1} : (C_b(Y),\beta) \rightarrow (C_b(X),\beta)$ is continuous.
\end{corollary}

\begin{corollary}[Open mapping theorem] \label{corollary:open_mapping}
Let $X,Y$ be separably metrizable spaces. Let $T : (C_b(X),\beta) \rightarrow (C_b(Y),\beta)$ be a surjective continuous linear map. Then $T$ is open.
\end{corollary}

\section{Identifying the finest topology coinciding with \texorpdfstring{$\sigma$}{sigma} on all \texorpdfstring{$\beta$}{beta} equi-continuous sets}

Denote by $\cM_{\tau,+}(X)$ the subset of non-negative $\tau$-additive Borel measures on $X$ and denote by $\sigma_+$ the restriction of $\sigma$ to $\cM_{\tau,+}(X)$. Consider the map
\begin{equation*}
\begin{cases}
q  : \cM_{\tau,+}(X) \times \cM_{\tau,+}(X) \rightarrow \cM_{\tau}(X) \\
q(\mu,\nu) = \mu - \nu.
\end{cases}
\end{equation*}
Note that by the Hahn-Jordan theorem the map $q$ is surjective.

\begin{definition}
Let $\cT$ denote the quotient topology on $\cM_\tau(X)$ of the map $q$ with respect to $\sigma_+ \times \sigma_+$ on $ \cM_{\tau,+}(X) \times \cM_{\tau,+}(X)$. 
\end{definition}

The next few lemma's will provide some key properties of $\cT$, which will lead to the proof that $\cT = \sigma^f$.

\begin{lemma} \label{lemma:tau_is_kspace}
$(\cM_\tau(X),\cT)$ is a k-space.
\end{lemma}

\begin{proof}
First of all, the topology $\sigma_+$ is metrizable by Theorem 8.3.2 in \cite{Bo07}. This implies that $\sigma_+^2$ is metrizable. Metrizable spaces are k-spaces by Theorem 3.3.20 in \cite{En89}. Thus $(\cM_\tau(X),\cT)$ is the quotient of a k-space which implies that $(\cM_\tau(X),\cT)$ is a k-space by Theorem 3.3.23 in \cite{En89}.
\end{proof}

\begin{lemma} \label{lemma:comparison_topologies_compacts}
The topology $\cT$ is stronger than $\sigma$. Both topologies have the same compact sets and on the compact sets the topologies agree. 
\end{lemma}

\begin{proof}
For $f \in C_b(X)$ denote $i_f : \cM_\tau(X) \rightarrow \bR$ defined by $i_f(\mu) = \int f \dd \mu$. As $\cT$ is the final topology under the map $q$, $i_f$ is continuous if and only if $i_f \circ q : \cM_{\tau,+}(X) \times \cM_{\tau,+}(X) \rightarrow \bR$ is continuous. This, however, is clear as $i_f \circ q (\mu,\nu) = \int f \dd (\mu - \nu)$ and the definition of the weak topology on $\cM_{\tau,+}(X)$.

\smallskip

$\sigma$ is the weakest topology making all $i_f$ continuous, which implies that $\sigma \subseteq \cT$.

\smallskip

For the second statement, note first that as $\sigma \subseteq \cT$, the first has more compact sets. Thus, suppose that $K \subseteq \cM_\tau(X)$ is $\sigma$ compact.  By Proposition \ref{proposition:Cb_strong_Mackey} $K$ is $\beta$ equi-continuous, so by Theorem 6.1 (c) in \cite{Se72}, $K \subseteq K_1 - K_2$, where $K_1,K_2 \subseteq \cM_{\tau,+}(X)$ and where $K_1,K_2$ are $\sigma_+$ and hence $\sigma$ compact. It follows that $q(K_1,K_2)$ is $\cT$ compact. As $K$ is a closed subset of $q(K_1,K_2)$, it is $\cT$ compact. We conclude that the $\sigma$ and $\cT$ compact sets coincide.

\smallskip

Let $K$ be a $\cT$ and $\sigma$ compact set. As the identity map $i : K \rightarrow K$ is $\cT$ to $\sigma$ continuous, it maps compacts to compacts. As all closed sets are compact, $i$ is homeomorphic, which implies that $\sigma$ and $\cT$ coincide on the compact sets.
\end{proof}

\begin{proposition} \label{proposition:tau_is_k_version_of_sigma}
$\cT$ is the k-ification of $\sigma$. In other words, $\cT$ is the finest topology that coincides with $\sigma$ on all $\sigma$ compact sets. In particular, we find that $\cT = \sigma^f$.
\end{proposition}

\begin{proof}
By Lemma \ref{lemma:tau_is_kspace}, $\cT$ is a k-space. By Lemma \ref{lemma:comparison_topologies_compacts} the compact sets for $\sigma$ and $\cT$ coincide. It follows that $\cT = k \sigma = \sigma^f$.
\end{proof}

We prove an additional lemma that will yield transseparability of $(C_b(X),\beta)$, before moving on to the study of the quotient topology $\cT$.

\begin{lemma} \label{lemma:compacts_metrizable}
The $\sigma$, or equivalently, $\cT$ compact sets in $\cM_\tau(X)$ are metrizable. 
\end{lemma}

\begin{proof}
Let $K$ be a $\sigma$ compact set in $\cM_\tau(X)$. In the proof of Lemma \ref{lemma:comparison_topologies_compacts}, we saw that $K \subseteq q(K_1,K_2)$, where $K_1,K_2$ are compact sets of the metrizable space $\cM_{\tau,+}(X)$. As $q$ is a continuous map, we find that $q(K_1,K_2)$ and hence $K$ is metrizable by Lemma 1.2 in \cite{Ka71} or 34.11.2 in \cite{Ko79}.
\end{proof}

\subsection{\texorpdfstring{$(\cM_\tau(X),\cT)$}{(Mt(X),T)} is a locally convex space.}

This section will focus on proving that the topology $\cT$ on $\cM_\tau(X)$ turns $\cM_\tau(X)$ into a locally convex space. Given the identification $\cT = k \sigma = \sigma^f$ obtained in Propositions \ref{proposition:Cb_strong_Mackey} and \ref{proposition:tau_is_k_version_of_sigma}, this is the main ingredient for the proof of Theorem \ref{theorem:main_theorem}. Indeed, for a general locally convex space the topology $\sigma^f$ is in general not a vector space topology, cf. Section 2 in \cite{Ko64}.

\begin{proposition} \label{proposition:tau_is_a_topological_vector_space}
$(\cM_\tau(X),\cT)$ is a topological vector space.
\end{proposition}

The proof of the proposition relies on two lemma's.

\begin{lemma} \label{lemma:q_is_open}
The map $q : (\cM_{\tau,+}(X)^2,\sigma^2_+) \rightarrow (\cM_\tau(X),\cT)$ is an open map.
\end{lemma}

\begin{proof}
Before we start proving that the map $q$ is open, we start with two auxiliary steps.

\textit{Step 1.} We first prove that the map $\oplus : (\cM_{\tau,+}^2(X) \times \cM_\tau(X),\sigma^2_+ \times \sigma) \rightarrow (\cM^2_\tau(X),\sigma^2)$, defined by $\oplus(\mu,\nu,\rho) = (\mu + \rho,\nu+ \rho)$ is open.

\smallskip

It suffices to show that $\oplus(V)$ is open for $V$ in a basis for $\sigma^2 \times \sigma$ by Theorem 1.1.14 in \cite{En89}. Hence, choose $A$ and $B$ be open for $\sigma_+$ and $C$ open for $\sigma$. Set $U := \oplus(A \times B \times C)$. Choose $(\mu,\nu) \in U$. We prove that there exists an open neighbourhood of $(\mu,\nu)$ contained in $U$. As $(\mu,\nu) \in U = \oplus(A \times B \times C)$, we find $\mu_0 \in A, \nu_0 \in B$ and $\rho_0 \in C$ such that $\mu = \mu_0 + \rho_0$ and $\nu = \nu_0 + \rho_0$.

As $\sigma$ is the topology of a topological vector space, the sets $\mu_0 + C$ and $\nu_0 + C$ are open for $\sigma$. Thus, the set $H:= (\mu_0 + C) \times (\nu_0 + C)$ is open for $\sigma^2$. By construction $(\mu,\nu) \in H$, and additionally, $H \subseteq U = \oplus(A \times B \times C)$.

We conclude that $\oplus$ is an open map.

\smallskip

\textit{Step 2.} Denote $G := \oplus^{-1}(\cM_{\tau,+}(X)^2)$ and by $\oplus_r : G \rightarrow \cM_{\tau,+}(X)^2$ the restriction of $\oplus$ to the inverse image of $\cM_{\tau,+}(X)^2$. If we equip $G$ with the subspace topology inherited from $(\cM_{\tau,+}^2(X) \times \cM_\tau(X),\sigma^2_+ \times \sigma)$, the map $\oplus_r$ is open by Proposition 2.1.4 in \cite{En89}  by the openness of $\oplus$.

\smallskip

\textit{Step 3: The proof that $q$ is open.} 

Let $V$ be an arbitrary open set in $(\cM_{\tau,+}(X)^2,\sigma_+^2)$. As a consequence, $V \times \cM_\tau(X)$ is open in $(\cM_{\tau,+}^2(X) \times \cM_\tau(X),\sigma^2_+ \times \sigma)$. By definition of the subspace topology, $(V \times \cM_\tau(X))\cap G$ is open for the subspace topology on $G$. By the openness of $\oplus_r$, we conclude that $\hat{V} := \oplus_r((V \times \cM_\tau(X))\cap G)$ is open in $(\cM_{\tau,+}(X)^2,\sigma_+^2)$.

\smallskip

As $\oplus_r((V \times \cM_\tau(X))\cap G) = \oplus(V \times \cM_\tau(X)) \cap \cM_{\tau,+}(X)^2$, we find that
\begin{align*} \label{eqn:representation_hat_V}
\hat{V} & = \left\{(\mu,\nu) \in \cM_{\tau,+}(X)^2 \, \middle| \, \exists \rho \in \cM_\tau(X): \, (\mu - \rho,\nu - \rho) \in V \right\} \\
& = \left\{(\mu,\nu) \in \cM_{\tau,+}(X)^2 \, \middle| \, \exists \rho \in \cM_\tau(X): \, (\mu + \rho,\nu + \rho) \in V \right\}.
\end{align*}
Thus, we see that $\hat{V} = q^{-1}(q(V))$. As $\hat{V}$ is open and $q$ is a quotient map, we obtain that $q(V)$ is open.
\end{proof}

\begin{lemma} \label{lemma:q_square_is_open}
The map $q^2 : (\cM_{\tau,+}(X)^4,\sigma^4_+) \rightarrow (\cM_\tau(X)^2,\cT^2)$, defined as the product of $q$ times $q$, i.e.
\begin{equation*}
q^2(\nu_1^+,\nu_1^-,\nu_2^+,\nu_2^-) = (\nu_1^+ - \nu_1^-,\nu_2^+ - \nu_2^-),
\end{equation*}
is an open map. As a consequence, $\cT^2$ is the quotient topology of $\sigma_+^4$ under $q^2$.
\end{lemma}

\begin{proof}
By Proposition 2.3.29 in \cite{En89} the product of open surjective maps is open. Thus, $q^2$ is open as a consequence of Lemma \ref{lemma:q_is_open}. An open surjective map is always a quotient map by Corollary 2.4.8 in \cite{En89}.
\end{proof}

\begin{proof}[Proof of Proposition \ref{proposition:tau_is_a_topological_vector_space}]
We start by proving that $(\cM_\tau(X)\times \cM_\tau(X),\cT^2) \rightarrow (\cM_\tau(X),\cT)$ defined by $+(\nu_1,\nu_2) = \nu_1 + \nu_2$ is continuous. Consider the following spaces and maps: 

\begin{tikzpicture}
  \matrix (m) [matrix of math nodes,row sep=3em,column sep=4em,minimum width=2em]
  {
     (\cM_\tau(X)\times \cM_\tau(X),\cT^2) & (\cM_\tau(X),\cT) \\
     (\cM_{\tau,+}(X)^4,\sigma^4) & (\cM_{\tau,+}(X)^2,\sigma_+^2) \\};
  \path[-stealth]
    (m-1-1) edge node [below] {$+$} (m-1-2)
    (m-2-1) edge node [left] {$q^2$} (m-1-1)
    (m-2-1) edge node [below] {$+_2$} (m-2-2)
    (m-2-2) edge node [right] {$q$} (m-1-2);
\end{tikzpicture}

$q$ and $+$ are the quotient and sum maps defined above. $q^2$ was introduced in Lemma \ref{lemma:q_square_is_open} and $+_2$ is defined as
\begin{equation*}
+_2(\nu_1^+,\nu_1^-,\nu_2^+,\nu_2^-) = (\nu_1^+ + \nu_2^+,\nu_1^- + \nu_2^-).
\end{equation*} 
Note that the diagram commutes, i.e. $q \circ +_2 = + \circ q^2$. 

Fix an open set $U$ in $(\cM_\tau(X),\cT)$, we prove that $+^{-1}(U)$ is $\cT^2$ open in $\cM_\tau(X) \times \cM_\tau(X)$. By construction, $q$ is continuous. Also, $+_2$ is continuous as it is the restriction of the addition map on a locally convex space. We obtain that $V:= +_2^{-1}(q^{-1}(U)) = (q^2)^{-1}(+^{-1}(U))$ is $\sigma^4_+$ open. By Lemma \ref{lemma:q_square_is_open} $q^2$ is a quotient map, which implies that $+^{-1}(U)$ is $\cT^2$ open. We conclude that $+ : (\cM_\tau(X)^2,\cT^2) \rightarrow (\cM_\tau(X),\cT)$ is continuous.

\smallskip

We proceed by proving that the scalar multiplication map $m : (\cM_\tau(X) \times \bR, \cT \times t) \rightarrow (\cM_\tau(X),\cT)$ defined by $m(\mu,\alpha) = \alpha \mu$ is continuous. Here, $t$ denotes the usual topology on $\bR$. Consider the following diagram:

\begin{tikzpicture}
  \matrix (m) [matrix of math nodes,row sep=3em,column sep=4em,minimum width=2em]
  {
     (\cM_\tau(X)\times \bR,\cT \times t) & (\cM_\tau(X),\cT) \\
     (\cM_{\tau,+}(X)^2 \times \bR,\sigma^2_+ \times t) & (\cM_\tau(X)^2,\sigma^2_+) \\};
  \path[-stealth]
    (m-1-1) edge node [below] {$m$} (m-1-2)
    (m-2-1) edge node [left] {$q \times I$} (m-1-1)
    (m-2-1) edge node [below] {$m_2$} (m-2-2)
    (m-2-2) edge node [right] {$q$} (m-1-2);
\end{tikzpicture}

Here, $I : \bR \rightarrow \bR$ denotes the identity map and $m_2 : \cM_{\tau,+}(X)^2 \times \bR \rightarrow \cM_{\tau,+}^2(X)$ is defined by
\begin{equation*}
m_2(\mu_1,\mu_2,\alpha)
\begin{cases}
(-\alpha \mu_2, - \alpha \mu_1) & \text{if } \alpha < 0 \\
(0,0) & \text{if } \alpha = 0 \\
(\alpha \mu_1,\alpha \mu_2) & \text{if } \alpha > 0.
\end{cases}
\end{equation*} 
Note that, using this definition of $m_2$, the diagram above commutes. It is straightforward to verify that $m_2$ is a $\sigma^2_+ \times t$ to $\sigma^2_+$ continuous map as $\sigma$ is the restriction of the topology of a topological vector space. By the Whitehead theorem, Theorem 3.3.7 in \cite{En89}, the map $q \times I$ is a quotient map. We obtain, as above, that $m$ is continuous. 
\end{proof}

The following proposition follows by 21.8.1 in \cite{Ko69} and Proposition \ref{proposition:tau_is_a_topological_vector_space}.

\begin{proposition} \label{proposition:tau_is_locally_convex}
$(\cM_\tau(X),\cT)$ has a basis of convex sets. As a consequence, $(\cM_\tau(X),\cT)$ is a locally convex space.
\end{proposition}

\subsection{The proof of Theorem \ref{theorem:main_theorem} and its corollaries}

We finalize with the proof of our main result and its consequences.

\begin{proof}[Proof of Theorem \ref{theorem:main_theorem}]
We already noted that $k\sigma = \sigma^f$ by Proposition \ref{proposition:Cb_strong_Mackey}.

By Proposition \ref{proposition:tau_is_k_version_of_sigma}, we find $\cT = \sigma^f$. By Proposition \ref{proposition:tau_is_locally_convex} $\cT$ is locally convex. As $\sigma^{lf}$ is the strongest locally convex topology coinciding with $\sigma$ on all weakly compact sets, we conclude by Proposition \ref{proposition:tau_is_locally_convex} that $\sigma^{lf} = \sigma^l$.

\smallskip

By Proposition \ref{proposition:hypercomplete_implied_by_sigma_is_sigmaf} the space $(C_b(X),\beta)$ is hypercomplete, and thus, complete. It follows by 21.9.8 in \cite{Ko69} that $\sigma^{lf} = \beta^\circ$.
\end{proof}

\begin{proof}[Proof of Lemma \ref{lemma:transseparable}]
As the $\sigma$ compact sets are metrizable by Lemma \ref{lemma:compacts_metrizable}, we find that $(C_b(X),\beta)$ is transseperable by Lemma 1 in \cite{Pf76}. 
\end{proof}

\begin{proof}[Proof of Lemma \ref{lemma:weak_sequential_completeness}]
The lemma follows immediately from Theorem 8.7.1 in \cite{Bo07}. A second proof can be given using the theory of Mazur spaces.

\smallskip

$\beta$ is the Mackey topology on $(C_b(X),\beta)$ by Proposition \ref{proposition:Cb_strong_Mackey}, we find $\cM_\tau(X)$ is weakly sequentially complete by Theorem 8.1 in \cite{Se72},  Theorem 7.4 in \cite{Wi81} and Propositions 4.3 and 4.4 in \cite{We68}.
\end{proof}

\begin{proof}[Proof of Corollary \ref{corollary:closed_graph}]
By Theorem \ref{theorem:main_theorem} and \ref{proposition:hypercomplete_implied_by_sigma_is_sigmaf}, we obtain that $(C_b(Y),\beta)$ is an infra-Pt\'{a}k space. By Lemma \ref{lemma:transseparable} $(C_b(X),\beta)$ is transseparable and by Lemma \ref{lemma:weak_sequential_completeness} $\cM_\tau(X)$ is weakly sequentially complete.

The result, thus, follows from Kalton's closed graph theorem \ref{theorem:Kalton}.
\end{proof}

\begin{proof}[Proof of Corollary \ref{corollary:inverse_mapping}]
Let $X, Y$ be seperably metrizable spaces. Let $T : (C_b(X),\beta) \rightarrow (C_b(Y),\beta)$ be a bijective continuous linear map. We prove that $T^{-1} : (C_b(Y),\beta) \rightarrow (C_b(X),\beta)$ is continuous.

The graph of a continuous map is always closed. Therefore, the graph of $T^{-1}$ is also closed. The result follows now from the closed graph theorem.
\end{proof}

\begin{proof}[Proof of Corollary \ref{corollary:open_mapping}]
Let $X,Y$ be seperably metrizable spaces. Let $T : (C_b(X),\beta) \rightarrow (C_b(Y),\beta)$ be a surjective continuous linear map. We prove that $T$ is open.

First, note that the quotient map $\pi : (C_b(X),\beta) \rightarrow (C_b(X) / ker \, T, \beta_{\pi})$ is open, where $\beta_{\pi}$ is the quotient topology obtained from $\beta$, see  15.4.2 \cite{Ko69}. The map $T$ factors into $T_\pi \circ \pi$, where $T_\pi$ is a bijective continuous linear map from $(C_b(X) / ker \, T, \beta_{\pi})$ to $(C_b(Y),\beta)$. 

We show that $T_\pi$ is an open map. We can apply the inverse mapping theorem to $T_\pi$ as $(C_b(X) /  ker \, T, \beta_{\pi})$ is a Pt\'{a}k space by 34.3.2 in \cite{Ko79}. Additionally, it is transseparable as it is the uniformly continuous image of a transseparable space. It follows that $T^{-1}_\pi$ is continuous and that $T_\pi$ is open.

We find that the composition $T = T_\pi \circ \pi$ is open as it is the composition of two open maps.
\end{proof}

\textbf{Acknowledgement}
The author is supported by The Netherlands Organisation for Scientific Research (NWO), grant number 600.065.130.12N109.

\bibliographystyle{plain} 
\bibliography{../KraaijBib}{}

\begin{thebibliography}{10}

\bibitem{Bo07}
Vladimir~I. Bogachev.
\newblock {\em Measure Theory}.
\newblock Springer-Verlag, 2007.

\bibitem{En89}
Ryszard Engelking.
\newblock {\em General topology}.
\newblock Heldermann Verlag, Berlin, second edition, 1989.

\bibitem{Ha76}
R.~G. Haydon.
\newblock On the {S}tone-{W}eierstrass theorem for the strict and superstrict
  topologies.
\newblock {\em Proc. Amer. Math. Soc.}, 59(2):273--278, 1976.

\bibitem{Ka71}
N.~J. Kalton.
\newblock Some forms of the {C}losed graph theorem.
\newblock {\em Mathematical Proceedings of the Cambridge Philosophical
  Society}, 70:401--408, 1971.

\bibitem{Ke58}
J.~L. Kelley.
\newblock Hypercomplete linear topological spaces.
\newblock {\em Michigan Math. J.}, 5(2):235--246, 1958.

\bibitem{Ko64}
Yukio K\={o}mura.
\newblock Some examples on linear topological spaces.
\newblock {\em Mathematische Annalen}, 153(2):150--162, 1964.

\bibitem{Ko69}
Gottfried K\"{o}the.
\newblock {\em Topological Vector Spaces I}.
\newblock Springer-Verlag, 1969.

\bibitem{Ko79}
Gottfried K\"{o}the.
\newblock {\em Topological Vector Spaces II}.
\newblock Springer-Verlag, 1979.

\bibitem{PCB87}
J.~Bonet P.~P\`{e}rez~Carreras.
\newblock {\em Barrelled Locally Convex Spaces}.
\newblock North-Holland Publishing Co., Amsterdam, 1987.

\bibitem{Pf76}
Helmut Pfister.
\newblock Bemerkungen zum satz \"{u}ber die separabilit\"{a}t der
  {F}r\'{e}chet-{M}ontel-r\"{a}ume.
\newblock {\em Archiv der Mathematik}, 27(1):86--92, 1976.

\bibitem{Se72}
F.~Dennis Sentilles.
\newblock Bounded continuous functions on a completely regular space.
\newblock {\em Trans. Amer. Math. Soc.}, 168:311--336, 1972.

\bibitem{We68}
J.~H. Webb.
\newblock Sequential convergence in locally convex spaces.
\newblock {\em Mathematical Proceedings of the Cambridge Philosophical
  Society}, 64:341--364, 4 1968.

\bibitem{Wi81}
Albert Wilansky.
\newblock Mazur spaces.
\newblock {\em {Int. J. Math. Math. Sci.}}, 4:39--53, 1981.

\end{thebibliography}

\end{document}